\documentclass[a4paper,10pt]{amsart}
\usepackage[utf8]{inputenc}
\usepackage{amsthm}
\usepackage{amsmath}
\usepackage{bm}
\usepackage{amsfonts}
\usepackage{amssymb}
\usepackage{mathtools}
\usepackage{mathrsfs}
\usepackage{setspace}
\usepackage{xcolor}
\usepackage{textgreek}
\usepackage{todonotes}
\usepackage{caption}
\usepackage{tikz}
\usepackage{pgfplots}
    \pgfplotsset{
        compat=1.15,
        width=8cm,
    }
\usepackage[a4paper, left=3cm, right=3cm, top=3cm, bottom=3cm]{geometry}
\usepackage{thmtools} 
\usepackage{bm}
\usepackage{keyval}

\usepackage[pdfdisplaydoctitle,colorlinks,breaklinks,urlcolor=blue,linkcolor=blue,citecolor=blue]{hyperref} 

\newtheorem{thm}{Theorem}[section]
\providecommand{\thmautorefname}{Theorem}
\newtheorem{definition}[thm]{Definition}
\newtheorem{cor}[thm]{Corollary}
\providecommand{\corautorefname}{Corollary}
\newtheorem{lem}[thm]{Lemma}
\providecommand{\lemautorefname}{Lemma}
\newtheorem{prop}[thm]{Proposition}
\providecommand{\propautorefname}{Proposition}

\theoremstyle{definition}
\newtheorem{example}{Example}[section]

\newcommand{\secautorefname}{Section}
\newcommand{\ssecautorefname}{Subsection}
\newcommand{\N}{\mathbb{N}}
\newcommand{\Z}{\mathbb{Z}}
\newcommand{\R}{\mathbb{R}}
\newcommand{\C}{\mathbb{C}}
\newcommand{\PP}{\mathbb{P}}
\newcommand{\T}{\mathbb{T}}
\newcommand{\E}[1]{\mathbb{E}\left[#1\right]}
\newcommand{\Ee}[1]{\mathbb{E}_\varepsilon\left[#1\right]}
\newcommand{\Prob}[1]{\mathbb{P}\left(#1\right)}
\newcommand{\Probe}[1]{\mathbb{P}_\varepsilon\left(#1\right)}
\newcommand{\dvg}{\mathord{{\rm div}}\,}
\newcommand{\dvgphi}{\mathord{{\rm div}}^\phi}
\newcommand{\dvgphien}{\mathord{{\rm div}}^{\phi_n}}
\newcommand{\dvgphin}{\mathord{{\rm div}}^{\phi_{n+1}}}\newcommand{\Dphi}{\Delta^{\hspace{-0.05cm}\phi}}
\newcommand{\Dphien}{\Delta^{\hspace{-0.05cm}\phi_n}}
\newcommand{\Dphin}{\Delta^{\hspace{-0.05cm}\phi_{n+1}}}
\newcommand{\nablaphi}{\nabla^{\phi}}
\newcommand{\nablaphien}{\nabla^{\phi_n}}
\newcommand{\nablaphin}{\nabla^{\phi_{n+1}}}
\newcommand{\curl}{\mathord{{\rm curl}}\,}
\newcommand{\curlphi}{\mathord{{\rm curl}}^\phi}
\newcommand{\curlphin}{\mathord{{\rm curl}}^{\phi_{n+1}}}
\newcommand{\rmb}[1]{\textcolor{blue}{#1}}
\newcommand{\rmkk}[1]{\textcolor{red}{#1}}
\newcommand{\rmm}[1]{\textcolor{magenta}{#1}}
\newcommand{\rmg}[1]{\textcolor[rgb]{0.50,0.25,0.00}{#1}}
\newcommand{\Var}[1]{\mathord{{\rm Var}}\left[#1\right]}
\newcommand{\Cov}[1]{\mathord{{\rm Cov}}\left(#1\right)}

\theoremstyle{remark}
\newtheorem{rmk}[thm]{Remark}

\numberwithin{equation}{section}

\def\Side{\ChessSide}
\newcommand\ChessBoxA{%
  {\fboxsep=0pt\fbox{\color{\ChessColori}\rule{\Side}{\Side}}}}
\newcommand\ChessBoxB{%
  {\fboxsep=0pt\fbox{\color{\ChessColorii}\rule{\Side}{\Side}}}}

\makeatletter
\newcommand\Row[1]{%
  \par\nobreak\nointerlineskip\vskip-\fboxrule%
  \@tfor\@tempa:=#1 \do {\csname ChessBox\@tempa\endcsname\kern-\fboxrule}}
\define@key{chessB}{side}{\def\ChessSide{#1}}
\define@key{chessB}{colori}{\def\ChessColori{#1}}
\define@key{chessB}{colorii}{\def\ChessColorii{#1}}
\setkeys{chessB}{
  side=0.7em,
  colori=black!70,
  colorii=white}
\makeatother

\title[Measure-preserving selection of characteristics]{On measure-preserving selection \\ of solutions of ODEs}

\author{Umberto Pappalettera}
\address[U. Pappalettera]{Fakult\"at f\"ur Mathematik, Universit\"at Bielefeld, D-33501 Bielefeld, Germany}
\email{upappale(at)math.uni-bielefeld.de}

\keywords{}
\date\today

\begin{document}

\begin{abstract}
For every $k \in \N$ and $\alpha \in (0,1)$ we construct a divergence-free $u \in C^k([0,T],C^\alpha(\T^d,\R^d))$, $d \geq 2$, such that there is no measurable selection of solutions of the ODE $\dot{X}_t = u(t,X_t)$ that preserves the Lebesgue measure.
\end{abstract}

\maketitle

\section{Introduction}

Let $\T^d := (\R/\Z)^d$ denote the $d$-dimensional torus, $d \geq 2$. 
Given $T>0$ and a measurable velocity field $u : [0,T] \times \T^d \to \R^d$ satisfying the incompressibility condition $\dvg u (t, \cdot) = 0$ in the sense of distributions, we are interested in the ODE
\begin{align} \label{eq:ODE}
X(t,x) = x + \int_0^t u(s,X(s,x)) ds,
\quad
x \in \T^d,\,
t \in [0,T].
\end{align} 

Equation \eqref{eq:ODE} can be uniquely solved for every initial condition $x \in \T^d$ when $u$ is smooth. 
In this case, looking at the evolution of the Jacobian determinant $J := \mbox{det}\,D_x X$
\begin{align*}
\partial_t J (t,x) = J (t,x)\, \dvg u (t,X(t,x)) = 0, 
\end{align*} 
we have that the map $X_t := X(t,\cdot)$ preserves the $d$-dimensional Lebesgue measure $\mathscr{L}^d$ on $\T^d$ for every $t \in [0,T]$, namely
\begin{align} \label{eq:meas_pres}
(X_t)_\sharp \mathscr{L}^d = \mathscr{L}^d.
\end{align}

Given a continuous velocity field $u$, solutions of \eqref{eq:ODE} exist for every initial condition $x \in \T^d$ by Peano Theorem, but are not necessarily unique. 
In this work we consider the following question: is there a Lebesgue-measurable way $x \mapsto X(\cdot,x)$ to bundle together solutions of \eqref{eq:ODE} for almost every initial condition $x \in \T^d$, so that the Lebesgue measure is preserved by $X_t$?
In other words, we are interested in the existence of a \emph{measure-preserving selection of characteristics}, according to the following:
\begin{definition}
Given a measurable $u : [0,T] \times \T^d \to \R^d$ with $\dvg u(t,\cdot) = 0$, we say that $X : [0,T] \times \T^d \to \T^d$ is a measure-preserving selection of characteristics if:
\begin{itemize}
\item[($i$)]
For every $t \in [0,T]$, $X(t,\cdot) : \T^d \to \T^d$ is Lebesgue measurable and \eqref{eq:meas_pres} holds;
\item[($ii$)]
For almost every $x \in \T^d$ the map $t \mapsto u(t,X(t,x))$ is in $L^1([0,T])$ and \eqref{eq:ODE} holds.
\end{itemize}
\end{definition}

In the rough setting, a classical result is due to Di Perna and Lions \cite{DiLi89}: under the additional assumptions of essential boundedness $u \in L^\infty_{t,x}$ and Sobolev regularity $u \in L^1_t W^{1,1}_x$, they show existence and uniqueness of a measure-preserving selection of characteristics $X$ which additionally satisfies the flow property $X(t+s,x)=X(t,X(s,x))$ for almost every $x \in \T^d$ and every $t,s,t+s \in [0,T]$. 
A successive extension to $u \in L^\infty_{t,x} \cap L^1_t BV_x$ fields is due to Ambrosio \cite{Am04}.

We point out that in \cite{DiLi89,Am04} uniqueness of a measure-preserving selection of characteristics is not a consequence of uniqueness of solutions to \eqref{eq:ODE} for a.e. starting point $x \in \T^d$. 
The latter requires stronger conditions on $u$, for instance $u \in L^\infty_t W^{1,r}_x$ with $r>d$ as proved in \cite[Corollary 5.2]{CaCr21}. 
On the other hand, in \cite{BrCoDL20} the authors give, for every $r < d$ and $p<\infty$, an example of $u \in C_t W^{1,r}_x \cap C_t L^p_x$ leading to non-unique solutions of \eqref{eq:ODE} for a positive measure set of starting points. A refinement of this result is presented in \cite{PiSo23}.
More recently, for every $r < d$ and $\alpha < 1$ Kumar \cite{Ku23+} further improved these results constructing a vector field $u \in C_t W^{1,r}_x \cap C^\alpha_{t,x}$ with non-unique solutions of \eqref{eq:ODE} for a full-measure set of starting points.

Notice that in the aforementioned works the theory of \cite{DiLi89,Am04} still applies. Therefore, \eqref{eq:meas_pres} can be seen as a selection criterium for solutions of \eqref{eq:ODE}.
However, as pointed out already in \cite[Section IV.2]{DiLi89}, \eqref{eq:meas_pres} might not single out a unique $X$ when $u$ has not a full (weak) spatial derivative.

Here we are mostly interested in the existence issue.
In general, understanding whether or not a measure-preserving selection of characteristics is possible requires a good understanding of the set of \emph{all} solutions of the ODE \eqref{eq:ODE}; but given a velocity field $u$ that is not weakly differentiable, non-uniqueness of \eqref{eq:ODE} can be extremely wild and hard to characterise.

Our main result is the following:
\begin{thm} \label{thm:main}
For every $k \in \N$ and $\alpha\in(0,1)$ there exists a velocity field $u \in C^k_t C^\alpha_x$, satisfying $\dvg u(t,\cdot) = 0$ in the sense of distributions, without any measure-preserving selection of characteristics.
\end{thm}

Therefore, the regularity threshold on $u$ that guarantees existence of a measure-preserving selection of characteristics is morally the same guaranteeing existence \emph{and} uniqueness (that is, one full weak spatial derivative).

The main ``obstruction'' to the existence is the fact that forward flows (call it $Y$) associated with incompressible velocity fields $v \in C^k_t C^\alpha_x$ do not need to be essentially injective in the sense of \cite{Am04}:
\begin{align*}
\int_{A_1} \delta_{Y(t,x)} dx
\perp
\int_{A_2} \delta_{Y(t,x)} dx
\end{align*}  
for every $t \in [0,T]$ and disjoint Lebesgue measurable sets $A_1, A_2 \subset \T^d$. 

Therefore, in order to prove \autoref{thm:main} we construct $v$ such that $Y_T : \T^d \to \T^d$ is exactly two-to-one on a large subset $A$ of its domain (in a proper measure theoretic sense). 
To see this, we use that the (unique) forward flow $Y$ is such that $Y(T,x)$ has an explicit expression when looking at the dyadic expansion of $x \in A \subset \T^2$.
Then, taking the backward velocity field $u = u(t,x) := -v(T-t,x)$, we have that any measurable selection of characteristics $X$ can not reach much more than ``half'' of the torus at time $T$, in particular $(X_T)_\sharp \mathscr{L}^d \neq \mathscr{L}^d$.   
More generally, for the $u$ we construct the absolute continuity $\mathscr{L}^d \ll (X_T)_\sharp \mathscr{L}^d$ fails for every measurable selection of characteristics $X$.
On the other hand, we are currently unable to say if $(X_T)_\sharp \mathscr{L}^d \ll \mathscr{L}^d $ holds; namely, if measurable selections of characteristics are Regular Lagrangian Flows as defined in \cite{Am04}. 


Let us close this introduction with the following observation. 
Let $\{\chi^\delta\}_{\delta \in (0,1)}$ be standard spatial mollifiers and consider the divergence-free velocity fields $u^\delta := u \ast \chi^\delta$, where $u$ is given by \autoref{thm:main}.  
By spatial smoothness, for every $\delta \in (0,1)$ the ODE \eqref{eq:ODE} with $u$ replaced by $u^\delta$ has a unique solution $X^\delta$, that additionally preserves the Lebesgue measure.
Since preserving the Lebesgue measure is stable with respect to convergence of the flows in $L^1_{t,x}$, as a consequence of \autoref{thm:main} we have:
\begin{cor}
The family $\{ X^\delta\}_{\delta \in (0,1)}$ is not strongly precompact in $L^1_{t,x}$.
\end{cor}
We point out that precompactness in $L^1_{t,x}$ holds true as soon as the family $\{u^\delta\}_{\delta \in (0,1)}$ is uniformly bounded in $L^\infty_{t,x} \cap W^{1,1}_{t,x}$ and the flows are uniformly nearly incompressible\footnote{Namely, there exists $C>0$ such that $C^{-1} \leq \mbox{det}\,D_x X^\delta \leq C$ for every $\delta \in (0,1)$. This condition is immediately implied by smoothness and $\dvg u^\delta = 0$.}: this is the content of Bressan's compactness conjecture, proved by Bianchini and Bonicatto in \cite{BiBo20}.

\section*{Acknowledgements}
The author is grateful to Paolo Bonicatto and Elia Bruè for the useful discussions on the topic, and in particular to the latter for pointing out the reference \cite{ElZl19}.
This project has received funding from the European Research Council (ERC) under the European Union’s Horizon 2020 research and innovation programme (grant agreement No. 949981).

\section{An explicit construction with $u \in L^\infty_{t,x}$}
\label{sec:bounded}
For the sake of a clear presentation, in this section we preliminarily construct an incompressible velocity field $u \in L^\infty_{t,x}$ such that the backward equation
\begin{align} \label{eq:back}
Y(t,x) = x - \int_0^t u(T-s,Y(s,x)) ds,
\quad
t \in [0,T],
\end{align}
has a unique solution for almost every initial condition $x \in \T^d$ and the map $Y_T : \T^d \to \T^d$ is exactly two-to-one, up to negligible sets. 

The idea is to let $Y_T$ essentially act on elements $x \in \T^d$ as a change of some (properly chosen) digits is the dyadic expansion of $x$, so that we can explicitly characterise almost every pair $(x,y) \in \T^d \times \T^d$ such that $Y_T(x) = y$.   

In the forthcoming \autoref{sec:holder}, we smooth out $u$ to obtain a new incompressible velocity field, with the desired H\"older regularity, such that the (unique) solution $Y^\delta$ of the backward equation \eqref{eq:back} is such that $Y^\delta(T,x) = Y(T,x)$ for every $x \in A \subset \T^d$, for some measurable $A$ with Lebesgue measure $|A| \geq 1-\delta$, with $\delta>0$ fixed but arbitrary.

For simplicity we restrict ourselves to $d=2$ but the construction generalizes easily to higher dimensions.

\subsection{Dyadic expansion} \label{ssec:dyadic}
For our purposes it is convenient to identify the Lebesgue measure on Borel subsets of two dimensional torus $\T^2$ with a probability measure on (pairs of) infinite sequences of digits $0,1$.
The following construction is more or less classical, and can be found for instance in \cite[p. 159]{Ha50}.

Let us consider the measurable space $(\Omega_0,\mathcal{F}_0)$ given by 
\begin{align*}
\Omega_0 := \{(0,0),(0,1),(1,0),(1,1)\},
\quad
\mathcal{F}_0 := 2^{\Omega_0},
\end{align*}
endowed with the uniform probability $\PP_0$ defined as $\PP_0(\{a\}):=1/4$ for every $a \in \Omega_0$. 
For every $n \in \N$, $n \geq 1$ consider an identical copy $(\Omega_n,\mathcal{F}_n,\PP_n)$ of $(\Omega_0,\mathcal{F}_0,\PP_0)$ and take the product space
\begin{align*}
\tilde{\Omega} := \prod_{n \geq 1} \Omega_n,
\quad
\tilde{\mathcal{F}} := \bigotimes_{n \geq 1} \mathcal{F}_n,
\quad
\tilde{\PP} := \bigotimes_{n \geq 1} \PP_n.
\end{align*}
Elements $\omega \in \tilde{\Omega}$ can be seen as pairs $\omega = (\omega^1,\omega^2)$ of infinite sequences of digits $0,1$.
In order to obtain a one-to-one correspondence with points of $\T^2$, we need to remove from $\tilde{\Omega}$ all those elements for which either $\omega^1$ or $\omega^2$ is a sequence definitely equal to $1$. More precisely, define for $i \in \{1,2\}$
\begin{align*}
\hat{\Omega}^i
&:=
\{ \omega = (\omega^1,\omega^2) \in \Omega :\,
\exists\, n_0 \in \N \mbox{ such that }
\omega^i_n = 1,\, 
\forall n > n_0
\}
\\
&=
\bigcup_{n_0 \in \N}
\bigcap_{n > n_0}
\{ \omega^i_n = 1 \},
\end{align*}
which are $\tilde{\mathcal{F}}$-measurable and $\tilde{\PP}$-negligible sets, and define
\begin{align*}
\Omega := \tilde{\Omega} \setminus (\hat{\Omega}^1 \cup \hat{\Omega}^2),
\quad
\mathcal{F} := \tilde{\mathcal{F}} \cap \Omega
:= \{ A \cap \Omega : A \in \tilde{\mathcal{F}} \},
\quad
\PP(A \cap \Omega) := \tilde{\PP}(A).
\end{align*} 

Then, the probability space $(\Omega,\mathcal{F},\PP)$ is isomorphic to the space of Borel subsets of $[0,1)^2$, endowed with the Lebesgue measure, via the bijection $x : \Omega \to [0,1)^2$ defined as
\begin{align} \label{eq:dyadic}
x(\omega) = (x^1,x^2),
\quad
x^1 := \sum_{n \geq 1} 2^{-n} \omega^1_n ,
\quad
x^2 := \sum_{n \geq 1} 2^{-n} \omega^2_n .
\end{align}
More precisely, $x$ is $\mathcal{F}$-$\mathcal{B}([0,1)^2)$ measurable, bijective, with measurable inverse $x^{-1}$ and we have the identity of the pushforward measures $(x)_\sharp \PP = \mathscr{L}^2$ and $(x^{-1})_\sharp \mathscr{L}^2 = \PP$. 
Identifying $\T^2 \sim [0,1)^2$, we call the representation above the \emph{dyadic expansion} of a point $x \in \T^2$.

\subsection{Backward Flow}
Let $\tau \in (0,1)$ be a time parameter and define 
\begin{align*}
\tau_0 := 0,
\quad
\tau_n := \sum_{q = 1}^n \tau^q,
\quad
n \in \N,\,
n \geq 1.
\end{align*}
Recall the dyadic expansion \eqref{eq:dyadic} of points $x=(x^1,x^2) \in \T^2$. 
We want the map $Y$ to act in the following way (say on a full-measure set of points):
\begin{itemize}
\item
From time $s=\tau_0$ to time $t=\tau_1$, if $\omega^2_2=0$ then $Y_{s,t}(x)=x$, and if $\omega^2_2=1$ then $Y_{s,t}(x)$ has the same dyadic expansion of $x$ except for a change in the digit $\omega^1_1$;
\item
From time $s=\tau_1$ to time $t=\tau_2$, if $\omega^1_3=0$ then $Y_{s,t}(x)=x$, and if $\omega^1_3=1$ then $Y_{s,t}(x)$ has the same dyadic expansion of $x$ except for a change in the digit $\omega^2_2$;
\item
\dots
\item
From time $s=\tau_{2n}$ to time $t=\tau_{2n+1}$, $n \in \N$, if $\omega^2_{2n+2}=0$ then $Y_{s,t}(x)=x$, and if $\omega^2_{2n+2}=1$ then $Y_{s,t}(x)$ has the same dyadic expansion of $x$ except for a change in the digit $\omega^1_{2n+1}$;
\item
From time $s=\tau_{2n+1}$ to time $t=\tau_{2n+2}$, $n \in \N$, if $\omega^1_{2n+3}=0$ then $Y_{s,t}(x)=x$, and if $\omega^1_{2n+3}=1$ then $Y_{s,t}(x)$ has the same dyadic expansion of $x$ except for a change in the digit $\omega^2_{2n+2}$.
\end{itemize}

In the lines above we have denoted $Y_{s,t}(x)$ the solution at time $t$ of \eqref{eq:back} passing through $x$ at time $s$ (we will see that it exists and is unique for every $s,t$ as above and almost every $x\in \T^2$).

Geometrically, the action of $Y$ is visualized in \autoref{fig:evolution}.
The evolution via $Y$ from time $s=\tau_{2n}$ to time $t=\tau_{2n+1}$ corresponds to a rigid translation (i.e. without rotations or deformations) of the set $\{ \omega^2_{2n+2}=1,\omega^1_{2n+1}=0\}$ into the set $\{ \omega^2_{2n+2}=1,\omega^1_{2n+1}=1\}$, and viceversa. 
Moreover, after time $t$ the digits $\omega^i_q$ in positions $q \leq 2n+1$ are kept fixed throughout the evolution.
A similar interpretation holds for the evolution from time $s=\tau_{2n+1}$ to time $t=\tau_{2n+2}$, with $\omega^2_{q}$ replaced by $\omega^1_{q+1}$ for every $q \leq 2n+2$, and $\omega^1_{2n+1}$ replaced by $\omega^2_{2n+2}$.

\begin{figure}[h]

\begin{minipage}{12cm}
\begin{tikzpicture}
\shade[left color=black,right color=white] (0,0) rectangle (4,4);

\draw (0,0) -- (0,4) -- (4,4) -- (4,0) -- (0,0) node[below right]{\hspace{1.4cm} $t = \tau_0$};

\draw [-to] (4.8,2) -- (5.2,2);

\shade[left color=black,right color=white] (6,0) rectangle (10,1);
\shade[left color=black,right color=white] (6,2) rectangle (10,3);
\shade[left color=black!50!white,right color=white] (6,1) rectangle (8,2);
\shade[left color=black!50!white,right color=white] (6,3) rectangle (8,4);
\shade[left color=black,right color=black!50!white] (8,1) rectangle (10,2);
\shade[left color=black,right color=black!50!white] (8,3) rectangle (10,4);

\draw (6,0) -- (6,4) -- (10,4) -- (10,0) -- (6,0) node[below right]{\hspace{1.4cm} $t = \tau_1$};

\draw [-to] (10.8,2) -- (11.2,2);
\end{tikzpicture}
\end{minipage}

\vspace{1cm}
\begin{minipage}{12cm}
\begin{tikzpicture}
\draw [-to] (-1.2,2) -- (-0.8,2);

\shade[left color=black,right color=black!87.5!white] (0,0) rectangle (0.5,1);
\shade[left color=black,right color=black!87.5!white] (0,2) rectangle (0.5,3);
\shade[left color=black,right color=black!87.5!white] (2,1) rectangle (2.5,2);
\shade[left color=black,right color=black!87.5!white] (2,3) rectangle (2.5,4);
\shade[left color=black!87.5!white,right color=black!75!white] (0.5,1) rectangle (1,2);
\shade[left color=black!87.5!white,right color=black!75!white] (0.5,3) rectangle (1,4);
\shade[left color=black!87.5!white,right color=black!75!white] (2.5,0) rectangle (3,1);
\shade[left color=black!87.5!white,right color=black!75!white] (2.5,2) rectangle (3,3);
\shade[left color=black!75!white,right color=black!62.5!white] (1,0) rectangle (1.5,1);
\shade[left color=black!75!white,right color=black!62.5!white] (1,2) rectangle (1.5,3);
\shade[left color=black!75!white,right color=black!62.5!white] (3,1) rectangle (3.5,2);
\shade[left color=black!75!white,right color=black!62.5!white] (3,3) rectangle (3.5,4);
\shade[left color=black!62.5!white,right color=black!50!white] (1.5,1) rectangle (2,2);
\shade[left color=black!62.5!white,right color=black!50!white] (1.5,3) rectangle (2,4);
\shade[left color=black!62.5!white,right color=black!50!white] (3.5,0) rectangle (4,1);
\shade[left color=black!62.5!white,right color=black!50!white] (3.5,2) rectangle (4,3);
\shade[left color=black!50!white,right color=black!37.5!white] (0,1) rectangle (0.5,2);
\shade[left color=black!50!white,right color=black!37.5!white] (0,3) rectangle (0.5,4);
\shade[left color=black!50!white,right color=black!37.5!white] (2,0) rectangle (2.5,1);
\shade[left color=black!50!white,right color=black!37.5!white] (2,2) rectangle (2.5,3);
\shade[left color=black!37.5!white,right color=black!25!white] (0.5,0) rectangle (1,1);
\shade[left color=black!37.5!white,right color=black!25!white] (0.5,2) rectangle (1,3);
\shade[left color=black!37.5!white,right color=black!25!white] (2.5,1) rectangle (3,2);
\shade[left color=black!37.5!white,right color=black!25!white] (2.5,3) rectangle (3,4);
\shade[left color=black!25!white,right color=black!12.5!white] (1,1) rectangle (1.5,2);
\shade[left color=black!25!white,right color=black!12.5!white] (1,3) rectangle (1.5,4);
\shade[left color=black!25!white,right color=black!12.5!white] (3,0) rectangle (3.5,1);
\shade[left color=black!25!white,right color=black!12.5!white] (3,2) rectangle (3.5,3);
\shade[left color=black!12.5!white,right color=white] (1.5,0) rectangle (2,1);
\shade[left color=black!12.5!white,right color=white] (1.5,2) rectangle (2,3);
\shade[left color=black!12.5!white,right color=white] (3.5,1) rectangle (4,2);
\shade[left color=black!12.5!white,right color=white] (3.5,3) rectangle (4,4);

\draw (0,0) -- (0,4) -- (4,4) -- (4,0) -- (0,0)
node[below right]{\hspace{1.4cm} $t = \tau_2$};

\draw [-to] (4.8,2) -- (5.2,2);

\shade[left color=black,right color=black!87.5!white] (6,0) rectangle (6.5,0.25);
\shade[left color=black,right color=black!87.5!white] (6,0.5) rectangle (6.5,0.75);
\shade[left color=black,right color=black!87.5!white] (6.5,0.25) rectangle (7,0.5);
\shade[left color=black,right color=black!87.5!white] (6.5,0.75) rectangle (7,1);

\shade[left color=black,right color=black!87.5!white] (6,2) rectangle (6.5,2.25);
\shade[left color=black,right color=black!87.5!white] (6,2.5) rectangle (6.5,2.75);
\shade[left color=black,right color=black!87.5!white] (6.5,2.25) rectangle (7,2.5);
\shade[left color=black,right color=black!87.5!white] (6.5,2.75) rectangle (7,3);

\shade[left color=black,right color=black!87.5!white] (8,1) rectangle (8.5,1.25);
\shade[left color=black,right color=black!87.5!white] (8,1.5) rectangle (8.5,1.75);
\shade[left color=black,right color=black!87.5!white] (8.5,1.25) rectangle (9,1.5);
\shade[left color=black,right color=black!87.5!white] (8.5,1.75) rectangle (9,2);

\shade[left color=black,right color=black!87.5!white] (8,3) rectangle (8.5,3.25);
\shade[left color=black,right color=black!87.5!white] (8,3.5) rectangle (8.5,3.75);
\shade[left color=black,right color=black!87.5!white] (8.5,3.25) rectangle (9,3.5);
\shade[left color=black,right color=black!87.5!white] (8.5,3.75) rectangle (9,4);

\shade[left color=black!87.5!white,right color=black!75!white] (6.5,1) rectangle (7,1.25);
\shade[left color=black!87.5!white,right color=black!75!white] (6.5,1.5) rectangle (7,1.75);
\shade[left color=black!87.5!white,right color=black!75!white] (6,1.25) rectangle (6.5,1.5);
\shade[left color=black!87.5!white,right color=black!75!white] (6,1.75) rectangle (6.5,2);

\shade[left color=black!87.5!white,right color=black!75!white] (6.5,3) rectangle (7,3.25);
\shade[left color=black!87.5!white,right color=black!75!white] (6.5,3.5) rectangle (7,3.75);
\shade[left color=black!87.5!white,right color=black!75!white] (6,3.25) rectangle (6.5,3.5);
\shade[left color=black!87.5!white,right color=black!75!white] (6,3.75) rectangle (6.5,4);

\shade[left color=black!87.5!white,right color=black!75!white] (8.5,0) rectangle (9,0.25);
\shade[left color=black!87.5!white,right color=black!75!white] (8.5,0.5) rectangle (9,0.75);
\shade[left color=black!87.5!white,right color=black!75!white] (8,0.25) rectangle (8.5,0.5);
\shade[left color=black!87.5!white,right color=black!75!white] (8,0.75) rectangle (8.5,1);

\shade[left color=black!87.5!white,right color=black!75!white] (8.5,2) rectangle (9,2.25);
\shade[left color=black!87.5!white,right color=black!75!white] (8.5,2.5) rectangle (9,2.75);
\shade[left color=black!87.5!white,right color=black!75!white] (8,2.25) rectangle (8.5,2.5);
\shade[left color=black!87.5!white,right color=black!75!white] (8,2.75) rectangle (8.5,3);

\shade[left color=black!75!white,right color=black!62.5!white] (7,0) rectangle (7.5,0.25);
\shade[left color=black!75!white,right color=black!62.5!white] (7,0.5) rectangle (7.5,0.75);
\shade[left color=black!75!white,right color=black!62.5!white] (7.5,0.25) rectangle (8,0.5);
\shade[left color=black!75!white,right color=black!62.5!white] (7.5,0.75) rectangle (8,1);

\shade[left color=black!75!white,right color=black!62.5!white] (7,2) rectangle (7.5,2.25);
\shade[left color=black!75!white,right color=black!62.5!white] (7,2.5) rectangle (7.5,2.75);
\shade[left color=black!75!white,right color=black!62.5!white] (7.5,2.25) rectangle (8,2.5);
\shade[left color=black!75!white,right color=black!62.5!white] (7.5,2.75) rectangle (8,3);

\shade[left color=black!75!white,right color=black!62.5!white] (9,1) rectangle (9.5,1.25);
\shade[left color=black!75!white,right color=black!62.5!white] (9,1.5) rectangle (9.5,1.75);
\shade[left color=black!75!white,right color=black!62.5!white] (9.5,1.25) rectangle (10,1.5);
\shade[left color=black!75!white,right color=black!62.5!white] (9.5,1.75) rectangle (10,2);

\shade[left color=black!75!white,right color=black!62.5!white] (9,3) rectangle (9.5,3.25);
\shade[left color=black!75!white,right color=black!62.5!white] (9,3.5) rectangle (9.5,3.75);
\shade[left color=black!75!white,right color=black!62.5!white] (9.5,3.25) rectangle (10,3.5);
\shade[left color=black!75!white,right color=black!62.5!white] (9.5,3.75) rectangle (10,4);

\shade[left color=black!62.5!white,right color=black!50!white] (7.5,1) rectangle (8,1.25);
\shade[left color=black!62.5!white,right color=black!50!white] (7.5,1.5) rectangle (8,1.75);
\shade[left color=black!62.5!white,right color=black!50!white] (7,1.25) rectangle (7.5,1.5);
\shade[left color=black!62.5!white,right color=black!50!white] (7,1.75) rectangle (7.5,2);

\shade[left color=black!62.5!white,right color=black!50!white] (7.5,3) rectangle (8,3.25);
\shade[left color=black!62.5!white,right color=black!50!white] (7.5,3.5) rectangle (8,3.75);
\shade[left color=black!62.5!white,right color=black!50!white] (7,3.25) rectangle (7.5,3.5);
\shade[left color=black!62.5!white,right color=black!50!white] (7,3.75) rectangle (7.5,4);

\shade[left color=black!62.5!white,right color=black!50!white] (9.5,0) rectangle (10,0.25);
\shade[left color=black!62.5!white,right color=black!50!white] (9.5,0.5) rectangle (10,0.75);
\shade[left color=black!62.5!white,right color=black!50!white] (9,0.25) rectangle (9.5,0.5);
\shade[left color=black!62.5!white,right color=black!50!white] (9,0.75) rectangle (9.5,1);

\shade[left color=black!62.5!white,right color=black!50!white] (9.5,2) rectangle (10,2.25);
\shade[left color=black!62.5!white,right color=black!50!white] (9.5,2.5) rectangle (10,2.75);
\shade[left color=black!62.5!white,right color=black!50!white] (9,2.25) rectangle (9.5,2.5);
\shade[left color=black!62.5!white,right color=black!50!white] (9,2.75) rectangle (9.5,3);

\shade[left color=black!50!white,right color=black!37.5!white] (6,1) rectangle (6.5,1.25);
\shade[left color=black!50!white,right color=black!37.5!white] (6,1.5) rectangle (6.5,1.75);
\shade[left color=black!50!white,right color=black!37.5!white] (6.5,1.25) rectangle (7,1.5);
\shade[left color=black!50!white,right color=black!37.5!white] (6.5,1.75) rectangle (7,2);

\shade[left color=black!50!white,right color=black!37.5!white] (6,3) rectangle (6.5,3.25);
\shade[left color=black!50!white,right color=black!37.5!white] (6,3.5) rectangle (6.5,3.75);
\shade[left color=black!50!white,right color=black!37.5!white] (6.5,3.25) rectangle (7,3.5);
\shade[left color=black!50!white,right color=black!37.5!white] (6.5,3.75) rectangle (7,4);

\shade[left color=black!50!white,right color=black!37.5!white] (8,0) rectangle (8.5,0.25);
\shade[left color=black!50!white,right color=black!37.5!white] (8,0.5) rectangle (8.5,0.75);
\shade[left color=black!50!white,right color=black!37.5!white] (8.5,0.25) rectangle (9,0.5);
\shade[left color=black!50!white,right color=black!37.5!white] (8.5,0.75) rectangle (9,1);

\shade[left color=black!50!white,right color=black!37.5!white] (8,2) rectangle (8.5,2.25);
\shade[left color=black!50!white,right color=black!37.5!white] (8,2.5) rectangle (8.5,2.75);
\shade[left color=black!50!white,right color=black!37.5!white] (8.5,2.25) rectangle (9,2.5);
\shade[left color=black!50!white,right color=black!37.5!white] (8.5,2.75) rectangle (9,3);

\shade[left color=black!37.5!white,right color=black!25!white] (6.5,0) rectangle (7,0.25);
\shade[left color=black!37.5!white,right color=black!25!white] (6.5,0.5) rectangle (7,0.75);
\shade[left color=black!37.5!white,right color=black!25!white] (6,0.25) rectangle (6.5,0.5);
\shade[left color=black!37.5!white,right color=black!25!white] (6,0.75) rectangle (6.5,1);

\shade[left color=black!37.5!white,right color=black!25!white] (6.5,2) rectangle (7,2.25);
\shade[left color=black!37.5!white,right color=black!25!white] (6.5,2.5) rectangle (7,2.75);
\shade[left color=black!37.5!white,right color=black!25!white] (6,2.25) rectangle (6.5,2.5);
\shade[left color=black!37.5!white,right color=black!25!white] (6,2.75) rectangle (6.5,3);

\shade[left color=black!37.5!white,right color=black!25!white] (8.5,1) rectangle (9,1.25);
\shade[left color=black!37.5!white,right color=black!25!white] (8.5,1.5) rectangle (9,1.75);
\shade[left color=black!37.5!white,right color=black!25!white] (8,1.25) rectangle (8.5,1.5);
\shade[left color=black!37.5!white,right color=black!25!white] (8,1.75) rectangle (8.5,2);

\shade[left color=black!37.5!white,right color=black!25!white] (8.5,3) rectangle (9,3.25);
\shade[left color=black!37.5!white,right color=black!25!white] (8.5,3.5) rectangle (9,3.75);
\shade[left color=black!37.5!white,right color=black!25!white] (8,3.25) rectangle (8.5,3.5);
\shade[left color=black!37.5!white,right color=black!25!white] (8,3.75) rectangle (8.5,4);

\shade[left color=black!25!white,right color=black!12.5!white] (7,1) rectangle (7.5,1.25);
\shade[left color=black!25!white,right color=black!12.5!white] (7,1.5) rectangle (7.5,1.75);
\shade[left color=black!25!white,right color=black!12.5!white] (7.5,1.25) rectangle (8,1.5);
\shade[left color=black!25!white,right color=black!12.5!white] (7.5,1.75) rectangle (8,2);

\shade[left color=black!25!white,right color=black!12.5!white] (7,3) rectangle (7.5,3.25);
\shade[left color=black!25!white,right color=black!12.5!white] (7,3.5) rectangle (7.5,3.75);
\shade[left color=black!25!white,right color=black!12.5!white] (7.5,3.25) rectangle (8,3.5);
\shade[left color=black!25!white,right color=black!12.5!white] (7.5,3.75) rectangle (8,4);

\shade[left color=black!25!white,right color=black!12.5!white] (9,0) rectangle (9.5,0.25);
\shade[left color=black!25!white,right color=black!12.5!white] (9,0.5) rectangle (9.5,0.75);
\shade[left color=black!25!white,right color=black!12.5!white] (9.5,0.25) rectangle (10,0.5);
\shade[left color=black!25!white,right color=black!12.5!white] (9.5,0.75) rectangle (10,1);

\shade[left color=black!25!white,right color=black!12.5!white] (9,2) rectangle (9.5,2.25);
\shade[left color=black!25!white,right color=black!12.5!white] (9,2.5) rectangle (9.5,2.75);
\shade[left color=black!25!white,right color=black!12.5!white] (9.5,2.25) rectangle (10,2.5);
\shade[left color=black!25!white,right color=black!12.5!white] (9.5,2.75) rectangle (10,3);

\shade[left color=black!12.5!white,right color=white] (7.5,0) rectangle (8,0.25);
\shade[left color=black!12.5!white,right color=white] (7.5,0.5) rectangle (8,0.75);
\shade[left color=black!12.5!white,right color=white] (7,0.25) rectangle (7.5,0.5);
\shade[left color=black!12.5!white,right color=white] (7,0.75) rectangle (7.5,1);

\shade[left color=black!12.5!white,right color=white] (7.5,2) rectangle (8,2.25);
\shade[left color=black!12.5!white,right color=white] (7.5,2.5) rectangle (8,2.75);
\shade[left color=black!12.5!white,right color=white] (7,2.25) rectangle (7.5,2.5);
\shade[left color=black!12.5!white,right color=white] (7,2.75) rectangle (7.5,3);

\shade[left color=black!12.5!white,right color=white] (9.5,1) rectangle (10,1.25);
\shade[left color=black!12.5!white,right color=white] (9.5,1.5) rectangle (10,1.75);
\shade[left color=black!12.5!white,right color=white] (9,1.25) rectangle (9.5,1.5);
\shade[left color=black!12.5!white,right color=white] (9,1.75) rectangle (9.5,2);

\shade[left color=black!12.5!white,right color=white] (9.5,3) rectangle (10,3.25);
\shade[left color=black!12.5!white,right color=white] (9.5,3.5) rectangle (10,3.75);
\shade[left color=black!12.5!white,right color=white] (9,3.25) rectangle (9.5,3.5);
\shade[left color=black!12.5!white,right color=white] (9,3.75) rectangle (9.5,4);

\draw (6,0) -- (6,4) -- (10,4) -- (10,0) -- (6,0)
node[below right]{\hspace{1.4cm} $t = \tau_3$};

\end{tikzpicture}
\end{minipage}

\caption{First three time steps of the evolution via $Y$: from time $s=\tau_0$ to time $t=\tau_1$; from time $s=\tau_1$ to time $t=\tau_2$; and from time $s=\tau_2$ to time $t=\tau_3$.}

\label{fig:evolution}
\end{figure}
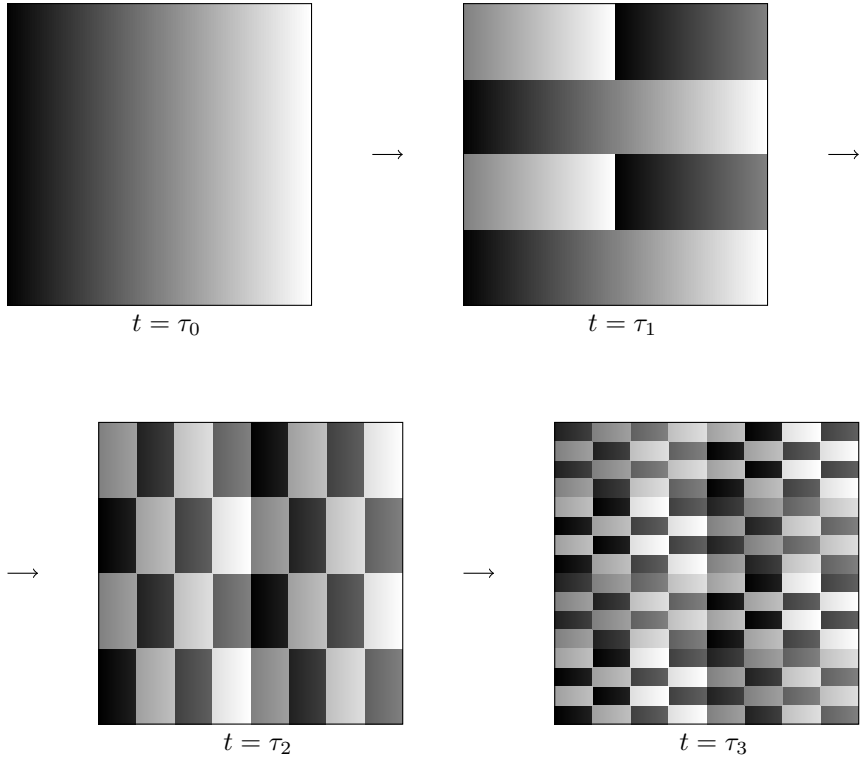

Denote $T:= \lim_{n \to \infty} \tau_n = \sum_{q=1}^\infty \tau^q$.
Algebraically, we can describe the action of $Y_T:\T^2 \to \T^2$ looking at the action of its conjugate $\Psi : \Omega \to \Omega$ via the isomorphism given by dyadic expansion. 
For our purposes, it is sufficient to describe the action of $\Psi$ on a full measure set $U$ invariant under $\Psi$, namely, such that $\Psi : U \to U$.

For $a \in \{0,1\}$, let us consider the following subsets of $\Omega$:
\begin{align*}
U^{odd,a}_{even}
&:=
\{ \omega = (\omega^1,\omega^2) \in \Omega
:\,
\exists\, n_0 \in \N \mbox{ such that }
\omega^1_{2n+2} = a,\, 
\forall n \geq n_0
\}
\\
&=
\bigcup_{n_0 \in \N}
\bigcap_{n \geq n_0}
\{ \omega^1_{2n+2} = a \},
\\
U^{even,a}_{odd}
&:=
\{ \omega = (\omega^1,\omega^2) \in \Omega
:\,
\exists\, n_0 \in \N \mbox{ such that }
\omega^2_{2n+1} = a,\, 
\forall n \geq n_0
\}
\\
&=
\bigcup_{n_0 \in \N}
\bigcap_{n \geq n_0}
\{ \omega^2_{2n+1} = a \}.
\end{align*}
Each of the previous sets is $\mathcal{F}$-measurable and $\PP$-negligible. In particular, they correspond to Borel subsets of $\T^2$ with zero Lebesgue measure, after identifying points of $\T^2$ with their dyadic expansion.
We define 
$$U := \Omega \setminus (U^{odd,0}_{even} \cup U^{odd,1}_{even} \cup U^{even,0}_{odd} \cup U^{even,1}_{odd}).$$

Elements of $U$ are those $\omega=(\omega^1,\omega^2)$ such that neither the sequence of digits $\omega^1_{2n+2}$ nor $\omega^2_{2n+1}$ is definitely constant.
For points $\omega = (\omega^1,\omega^2) \in U$, denote
\begin{align*}
I := &\{ (i,n) : i+n \mbox{ even},\, n \geq 2,\,\omega^i_n = 1\},
\\
I^- := &\{ (j,m) : \exists\, (i,n) \in I \mbox{ such that } j \neq i,\, m = n-1,\,  \}.
\end{align*}
Then $\Psi(\omega) = \Psi(\omega^1,\omega^2)$ is given in coordinates as
\begin{align*}
[\Psi(\omega)]^i_n =
\begin{cases}
\omega^i_n, &\mbox{ if } (i,n) \notin I^-,
\\
\omega^i_n+1 \,(\mbox{mod}\, 2), &\mbox{ if } (i,n) \in I^-.
\end{cases} 
\end{align*}
It is immediate to check $\Psi(\omega) \in U$ for every $\omega \in U$, since digits of the form $\omega^{1}_{2n+2}$ and $\omega^{2}_{2n+1}$ remain unchanged by definition.

\begin{example} \label{ex:1}
For instance, the particular element $$\omega=
\left( \begin{matrix}
\omega^1_1 & \omega^1_2 & \omega^1_3 & \omega^1_4 & \omega^1_5 & \omega^1_6 & \dots
\\
\omega^2_1 & \omega^2_2 & \omega^2_3 & \omega^2_4 & \omega^2_5 & \omega^2_6 & \dots
\end{matrix}
\right)
=
\left( \begin{matrix}
0 & 1 & 1 & 1 & 0 & 1 & \dots
\\
1 & 1 & 0 & 1 & 0 & 0 & \dots
\end{matrix}
\right)$$
evolves in the following way via the action of $\Psi$:
\begin{gather*}
\left( \begin{matrix}
{0} & 1 & 1 & 1 & 0 & 1 & \dots
\\
1 & 1 & 0 & 1 & 0 & 0 & \dots
\end{matrix}
\right)
\\
\downarrow \quad {\omega^2_2 = 1}
\\
\left( \begin{matrix}
{1} & 1 & 1 & 1 & 0 & 1 & \dots
\\
1 & {1} & 0 & 1 & 0 & 0 & \dots
\end{matrix}
\right)
\\
\downarrow \quad {\omega^1_3 = 1}
\\
\left( \begin{matrix}
1 & 1 & 1 & 1 & 0 & 1 & \dots
\\
1 & {0} & 0 & 1 & 0 & 0 & \dots
\end{matrix}
\right)
\\
\downarrow \quad \omega^2_4 = 1
\\
\left( \begin{matrix}
1 & 1 & 0 & 1 & 0 & 1 & \dots
\\
1 & 0 & 0 & 1 & 0 & 0 & \dots
\end{matrix}
\right)
\\
\downarrow \quad \small {\small \omega^1_5 = 0}
\\
\left( \begin{matrix}
1 & 1 & 0 & 1 & 0 & 1 & \dots
\\
1 & 0 & 0 & 1 & 0 & 0 & \dots
\end{matrix}
\right)
\\
\downarrow \quad \omega^2_6 = 0
\\
\left( \begin{matrix}
1 & 1 & 0 & 1 & 0 & 1 & \dots
\\
1 & 0 & 0 & 1 & 0 & 0 & \dots
\end{matrix}
\right)
\\
\downarrow \quad \dots
\end{gather*}
\end{example}

The restriction to $U \subset \Omega$ guarantees that the map $\Psi : U \to U$ is surjective and exactly two-to-one, that is: for every $\bar{\omega} \in U$ there are exactly two $\omega,\omega' \in U$ such that $\Psi (\omega) = \Psi (\omega') = \bar{\omega}$.
This is because, if we want to solve $\Psi(\omega) = \bar{\omega}$, then we can arbitrarily choose either $\omega^1_1=0$ or $\omega^1_1=1$, and then define the other digits $\omega^i_q$ accordingly. 
This procedure yields an element of $U$ since we necessarily have $\omega^1_{2n+2}  = \bar{\omega}^1_{2n+2}$ and $\omega^2_{2n+1}  = \bar{\omega}^2_{2n+1}$.

In addition, the two solutions $\omega,\omega' \in U$ are related to each other by the relation $\omega' = \sigma (\omega)$, where $\sigma:U \to U$ is the ``digits change'' map:
\begin{align} \label{eq:sigma}
[\sigma(\omega)]^i_n =
\begin{cases}
\omega^i_n, &\mbox{ if } i + n \mbox{ odd},
\\
\omega^i_n+1 \,(\mbox{mod}\, 2), &\mbox{ if } i + n \mbox{ even}.
\end{cases} 
\end{align}

\begin{lem} \label{lem:2to1}
For every $\bar{\omega} \in U$ there are exactly two $\omega,\omega' \in U$ such that $\Psi (\omega) = \Psi (\omega') = \bar{\omega}$.
Moreover, $\omega' = \sigma (\omega)$.
\end{lem}
\begin{proof}
Let us observe preliminarily that, by the very definition of $\Psi$ and $\sigma$,
\begin{align} \label{eq:Psi_sigma}
\Psi (\omega)
=
\Psi \sigma(\omega),
\quad
\forall \omega \in U.
\end{align}
Therefore, in order to prove the lemma it is sufficient to show that for every $\bar{\omega} \in U$ there is exactly one $\omega \in U$ such that $\Psi (\omega) = \bar{\omega}$ and $\omega^1_1=0$. 
Since $\Psi$ leaves unchanged any digit of the form $\omega^1_{2n+2}$ and $\omega^2_{2n+1}$, we necessarily have $\omega^1_{2n+2} = \bar{\omega}^1_{2n+2}$ and $\omega^2_{2n+1} = \bar{\omega}^2_{2n+1}$ for every $n \in \N$. 

Since we are assuming $\omega^1_1=0$, then the other digits of $\omega$ are inductively determined by 
\begin{align*}
\omega^2_{2n+2} =
\begin{cases}
1, \mbox{ if } \omega^1_{2n+1} \neq \bar{\omega}^1_{2n+1},
\\
0, \mbox{ if } \omega^1_{2n+1} = \bar{\omega}^1_{2n+1},
\end{cases}
\end{align*}
and
\begin{align*}
\omega^1_{2n+3} =
\begin{cases}
1, \mbox{ if } \omega^2_{2n+2} \neq \bar{\omega}^2_{2n+2},
\\
0, \mbox{ if } \omega^2_{2n+2} = \bar{\omega}^2_{2n+2}.
\end{cases}
\end{align*}
\end{proof}

\begin{example} \label{ex:2}
For $\bar{\omega}=\left( \begin{matrix}
0 & 1 & 1 & 1 & 0 & 1 & \dots
\\
1 & 1 & 0 & 1 & 0 & 0 & \dots
\end{matrix}
\right)$ 
as in \autoref{ex:1}, the following are the only two solutions of $\Psi(\omega)=\bar{\omega}$:
\begin{gather*}
\begin{array}{ccc}
\left( \begin{matrix}
0 & 1 & * & 1 & * & 1 & \dots
\\
1 & * & 0 & * & 0 & * & \dots
\end{matrix}
\right)
&
\quad
&
\left( \begin{matrix}
1 & 1 & * & 1 & * & 1 & \dots
\\
1 & * & 0 & * & 0 & * & \dots
\end{matrix}
\right)
\\
\downarrow \quad \omega^2_2 = 0
&
\quad
&
\downarrow \quad \omega^2_2 = 1
\\
\left( \begin{matrix}
0 & 1 & * & 1 & * & 1 & \dots
\\
1 & 0 & 0 & * & 0 & * & \dots
\end{matrix}
\right)
&
\quad
&
\left( \begin{matrix}
1 & 1 & * & 1 & * & 1 & \dots
\\
1 & 1 & 0 & * & 0 & * & \dots
\end{matrix}
\right)
\\
\downarrow \quad \omega^1_3 = 1
&
\quad
&
\downarrow \quad \omega^1_3 = 0
\\
\left( \begin{matrix}
0 & 1 & 1 & 1 & * & 1 & \dots
\\
1 & 0 & 0 & * & 0 & * & \dots
\end{matrix}
\right)
&
\quad
&
\left( \begin{matrix}
1 & 1 & 0 & 1 & * & 1 & \dots
\\
1 & 1 & 0 & * & 0 & * & \dots
\end{matrix}
\right)
\\
\downarrow \quad \omega^2_4 = 0
&
\quad
&
\downarrow \quad \omega^2_4 = 1
\\
\left( \begin{matrix}
0 & 1 & 1 & 1 & * & 1 & \dots
\\
1 & 0 & 0 & 0 & 0 & * & \dots
\end{matrix}
\right)
&
\quad
&
\left( \begin{matrix}
1 & 1 & 0 & 1 & * & 1 & \dots
\\
1 & 1 & 0 & 1 & 0 & * & \dots
\end{matrix}
\right)
\\
\downarrow \quad \omega^1_5 = 1
&
\quad
&
\downarrow \quad \omega^1_5 = 0
\\
\left( \begin{matrix}
0 & 1 & 1 & 1 & 1 & 1 & \dots
\\
1 & 0 & 0 & 0 & 0 & * & \dots
\end{matrix}
\right)
&
\quad
&
\left( \begin{matrix}
1 & 1 & 0 & 1 & 0 & 1 & \dots
\\
1 & 1 & 0 & 1 & 0 & * & \dots
\end{matrix}
\right)
\\
\downarrow \quad \omega^2_6 = 1
&
\quad
&
\downarrow \quad \omega^2_6 = 0
\\
\left( \begin{matrix}
0 & 1 & 1 & 1 & 1 & 1 & \dots
\\
1 & 0 & 0 & 0 & 0 & 1 & \dots
\end{matrix}
\right)
&
\quad
&
\left( \begin{matrix}
1 & 1 & 0 & 1 & 0 & 1 & \dots
\\
1 & 1 & 0 & 1 & 0 & 0 & \dots
\end{matrix}
\right)
\\
\downarrow \quad \dots
&
\quad
&
\downarrow \quad \dots
\end{array}
\end{gather*}
\end{example}

Let us conclude this subsection with the following lemma.
\begin{lem} \label{lem:measure}
$\sigma : U \to U$ is measure preserving.
\end{lem}
\begin{proof}
Since $U \subset \Omega$ has full $\PP$-measure and $\Omega \subset \tilde{\Omega}$ has full $\tilde{\PP}$-measure, it is sufficient to prove that $\tilde{\sigma}:\tilde{\Omega} \to \tilde{\Omega}$ defined on the whole $\tilde{\Omega}$ as in \eqref{eq:sigma} preserves $\tilde{\PP}$.
The measures $\tilde{\PP}$ and $(\tilde{\sigma}^{-1})_\sharp \tilde{\PP}$ coincide on the cylinder subsets
\begin{align*}
\mathcal{C} := \{ C \subset \tilde{\Omega} : 
C = \prod_{n \geq 1} C_n , \,
C_n \subset \Omega_n\, \forall n,\,
\exists\, n_0 \in \N \mbox{ such that } C_n = \Omega_n\, 
\forall n > n_0
\}. 
\end{align*}
Since $\mathcal{C}$ is a $\pi$-system generating the $\sigma$-field $\tilde{\mathcal{F}}$, we have $\tilde{\PP}=(\tilde{\sigma}^{-1})_\sharp \tilde{\PP}$. 
Therefore, for every measurable $A \subset U$
\begin{align*}
\PP (\sigma^{-1}(A))
=
\PP (\tilde{\sigma}^{-1}(A) \cap U)
=
\tilde{\PP} (\tilde{\sigma}^{-1}(A))
=
\tilde{\PP} (A)
=
\PP (A).
\end{align*}
\end{proof}

\subsection{The velocity field}
Let us recall the following construction by Depauw \cite{De03}.
Denote $Q = (0,1)^2$ the unit square. For $x \in Q$ let 
\begin{align*}
\psi_\star(x) := -8 \max \{ |x^1-1/2|^2 , |x^2-1/2|^2 \},
\quad
v_\star := \nabla^\perp \psi.
\end{align*}
The velocity $v_\star$ is BV with null divergence, and rotates the square $Q$ clockwise of an angle of $\pi$ radiants in time $t=1/2$.
Then we define $v = v(t,x) := -u(T-t,x)$ at any time $t \in [\tau_{n-1},\tau_{n})$, $n \in \N$, $n \geq 1$, as a collection of many properly rescaled/reflected copies of $v_\star$, glued together in a self similar fashion.

For instance, to obtain that on the time interval $[\tau_0,\tau_1)=[0,\tau)$ the resulting $Y=Y_{0,\tau}$ is a translation of $\{\omega^2_2=1,\omega^1_1=0\}$ into $\{\omega^2_2=1,\omega^1_1=1\}$ and viceversa (up to negligible sets), we can use first a periodic Depauw velocity field with ``cells'' $\{\omega^2_2=1\}$ (see \autoref{fig:Depauw}) and intensity proportional to $1/\tau$, so that after time $\tau/2$ the set $\{\omega^2_2=1,\omega^1_1=0\}$ goes into $\{\omega^2_2=1,\omega^1_1=1\}$ and viceversa, but with a rotation of $\pi$ radiants; and from time $\tau/2$ to time $\tau$ we use distinct Depauw velocity fields with ``cells'' $\{\omega^2_2=1,\omega^1_1=0\}$ and $\{\omega^2_2=1,\omega^1_1=1\}$, with intensity proportional to $1/\tau$, that rotate back of $\pi$ radiants each set separately (the composition of two rotations with different center and opposite angle is a rigid translation).
This is similar to what done by Zizza in \cite{Zi22}.
Notice that essential uniqueness of $Y$ descends from the $L^\infty \cap BV$ regularity of $v_\star$. 

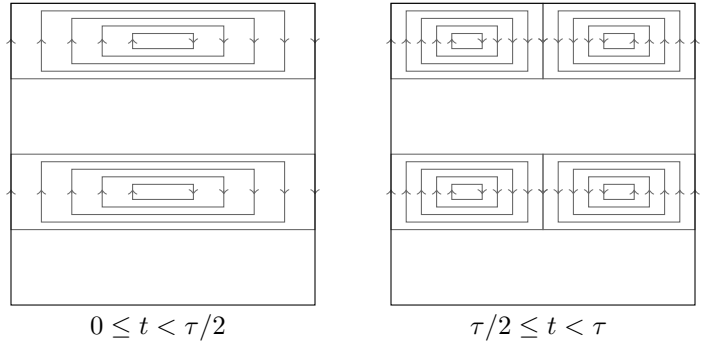
\begin{figure}[h] 
\begin{tikzpicture}

\foreach \i in {1,...,5} {
    \draw[color=black!60!white] (0.4*\i, 0.1*\i) rectangle (-0.4*\i, -0.1*\i);
    
    \draw[color=black!60!white,->] (0.4*\i,0.02) -- (0.4*\i,-0.03);
    \draw[color=black!60!white,->] (-0.4*\i,-0.02) -- (-0.4*\i,0.03);
    
    \draw[color=black!60!white] (0.4*\i, 0.1*\i-2) rectangle (-0.4*\i, -0.1*\i-2);
    
    \draw[color=black!60!white,->] (0.4*\i,0.02-2) -- (0.4*\i,-0.03-2);
    \draw[color=black!60!white,->] (-0.4*\i,-0.02-2) -- (-0.4*\i,0.03-2);
}

\draw (2,0.5) rectangle (-2,-3.5) node[below right]{\hspace{0.8cm} $0 \leq t < \tau/2$};

\foreach \i in {1,...,5} {
    \draw[color=black!60!white] (4+0.2*\i, 0.1*\i) rectangle (4-0.2*\i, -0.1*\i);
    
    \draw[color=black!60!white,->] (4+0.2*\i,0.02) -- (4+0.2*\i,-0.03);
    \draw[color=black!60!white,->] (4-0.2*\i,-0.02) -- (4-0.2*\i,0.03);
    
    \draw[color=black!60!white] (4+0.2*\i, 0.1*\i-2) rectangle (4-0.2*\i, -0.1*\i-2);
    
    \draw[color=black!60!white,->] (4+0.2*\i,0.02-2) -- (4+0.2*\i,-0.03-2);
    \draw[color=black!60!white,->] (4-0.2*\i,-0.02-2) -- (4-0.2*\i,0.03-2);
}

\foreach \i in {1,...,5} {
    \draw[color=black!60!white] (6+0.2*\i, 0.1*\i) rectangle (6-0.2*\i, -0.1*\i);
    
    \draw[color=black!60!white,->] (6+0.2*\i,-0.02) -- (6+0.2*\i,0.03);
    \draw[color=black!60!white,->] (6-0.2*\i,0.02) -- (6-0.2*\i,-0.03);
    
    \draw[color=black!60!white] (6+0.2*\i, 0.1*\i-2) rectangle (6-0.2*\i, -0.1*\i-2);
    
    \draw[color=black!60!white,->] (6+0.2*\i,-0.02-2) -- (6+0.2*\i,0.03-2);
    \draw[color=black!60!white,->] (6-0.2*\i,0.02-2) -- (6-0.2*\i,-0.03-2);
}

\draw (5+2,0.5) rectangle (5-2,-3.5)node[below right]{\hspace{0.8cm} $\tau/2 \leq t < \tau$};
\end{tikzpicture}
\caption{Velocity field $v_1$ for times $t \in [0,\tau/2)$ (left) and $t \in [\tau/2,\tau)$ (right). 
For times $t \in [0,\tau/2)$, the cells of the velocity field are given by $\{\omega^2_1=1,\omega^2_2=1\}$ (top) and $\{\omega^2_1=0,\omega^2_2=1\}$ (bottom).
For times $t \in [\tau/2,\tau)$, the cells of the velocity field are given by $\{\omega^2_1=1,\omega^2_2=1,\omega^1_1=0\}$ (top-left), $\{\omega^2_1=1,\omega^2_2=1,\omega^1_1=1\}$ (top-right), $\{\omega^2_1=0,\omega^2_2=1,\omega^1_1=0\}$ (bottom-left), and $\{\omega^2_1=0,\omega^2_2=1,\omega^1_1=1\}$ (bottom-right).
In both cases, the velocity is zero on the set $\{\omega^2_2=0\}$.}
\label{fig:Depauw}
\end{figure}

Up to rotations of $\pi/2$ radiants, the velocity field will be self-similar: on the time interval $t \in [\tau_{n-1},\tau_{n})$, $n \in \N$, $n \geq 1$, the velocity $v=v_n$ is given by
\begin{align*}
v_n(t,x) = \frac{1}{(2\tau)^n} v_1((t-\tau_{n})/\tau_{n} , 2^{n}x),
\end{align*}
where $v_1$ is the velocity field on the first time interval $[0,\tau)$.
Then, taking $\tau = 1/2$ we get
\begin{align*}
\| v_n \|_{L^\infty_{t,x}}
=
\| v_1 \|_{L^\infty_{t,x}}
< \infty
\end{align*}
uniformly in $n$, implying $\| v \|_{L^\infty_{t,x}} = \| u \|_{L^\infty_{t,x}} < \infty$.

\section{The H\"older continuous case}
\label{sec:holder}

In this section we present a H\"older continuous version of the velocity field from \autoref{sec:bounded} which has the property that the induced flow $Y^\delta_T$ coincides with $Y_T$ on a set $A$ of Lebesgue measure $|A|\geq 1-\delta$, for $\delta >0$ fixed but arbitrary. 
In order to do this, we have to replace the Depauw cells with more regular building blocks, that we describe hereafter.

For the sake of simplicity, we shall restrict ourself to constructing our building blocks on the unit square $Q=(0,1)^2$; the general case in handled by elementary operations as scalings, reflections, and rotations.

Let us recall the following construction from \cite{ElZl19}: define the function $\psi : Q \to \R$ as
\begin{align*}
\psi (x)
:=
2^\frac12 \frac{\sin(\pi x^1)\sin(\pi x^2)}{(\sin(\pi x^1)+\sin(\pi x^2))^\frac12}.
\end{align*}

Contour lines of $\psi$ foliate the unit square $Q$, as can be seen in \autoref{fig:v}. 
The value of $\psi$ ranges between $1$ (at the center of the square) and $0$ (in the limit when $x$ tends to the edges of the square).
The level set $\{ \psi \geq r \}$ is convex for every value of $r \in (0,1]$, see \cite[Lemma 3.1]{ElZl19}. 

Contour lines of $\psi$ are integral curves for the divergence-free velocity field $\nabla^\perp \psi$, and the time it takes a point $x \in \{\psi = r\}$ to run across its level set when moving with velocity $\nabla^\perp \psi$ is given by the line integral
\begin{align} \label{eq:time}
T_{\psi}(r) 
:=
\int_{\{\psi = r\}} \frac{1}{|\nabla \psi|} d\gamma.
\end{align}
\begin{figure}[h]
\captionsetup{justification=centering}
\includegraphics[width=4cm,height=4cm]{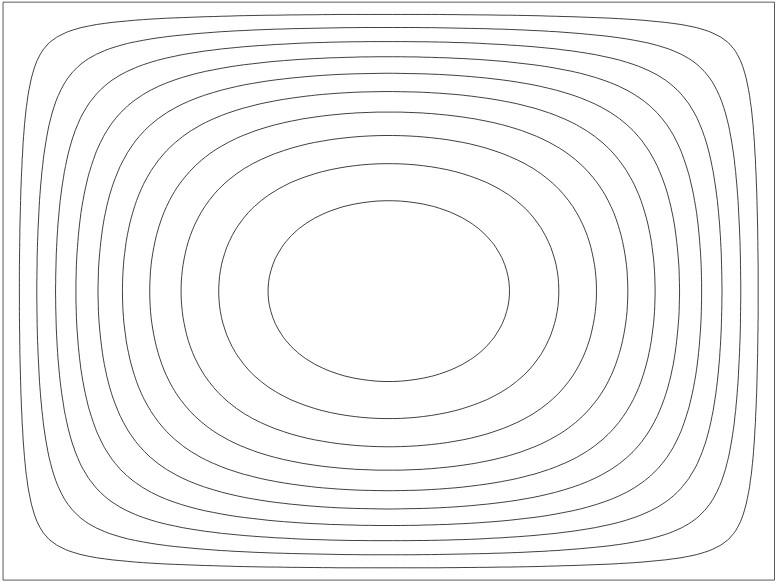}

\caption{Contour lines of $\psi$.}
\label{fig:v}
\end{figure}

Similarly to what done in \cite{ElZl19}, we want to replace $\psi$ with another streamfunction for which \eqref{eq:time} is independent of $r$. However, in order to obtain a velocity field which is globally H\"older continuous we need to modify the construction of \cite{ElZl19} as follows.

Let $\rho : \R_+ \to \R_+$ be a smooth function such that $\rho(r)=r$ for $r \geq 1$ and $\rho(r) = 1/2$ for $r \leq 1/2$. Denote $\rho^\epsilon := \epsilon \rho (\epsilon^{-1} r)$, where $\epsilon>0$ is a free parameter to be properly chosen later, and let
\begin{align*}
\psi^{\epsilon}(x)
:=
\int_0^{\rho^\epsilon(\psi(x))}
T_{\psi}(r) dr.
\end{align*}

\begin{lem} \label{lem:psieps}
There exists a constant $C \in (0,\infty)$, independent of $\epsilon$, such that the following hold:
\begin{itemize}
\item[($i$)] $\psi^\epsilon >0$ on $Q$, and $\psi^\epsilon$ is constant in a neighbourhood of $\partial Q$;
\item[($ii$)]
Let $r_\epsilon:=\int_0^\epsilon T_{\psi}(r) dr$.
Then for every $r \geq r_\epsilon$ the level set $\{\psi^\epsilon = r\}$ is also al level set for $\psi$, namely there exists $r'$ such that $\{\psi^\epsilon = r\}=\{\psi = r'\}$.
\item[($iii$)]
The travelling time $T_{\psi^\epsilon}$, defined as in \eqref{eq:time} with $\psi$ replaced by $\psi^\epsilon$, satisfies $T_{\psi^\epsilon} (r) = 1$ for all $r \in [r_\epsilon, \|\psi^\epsilon\|_{L^\infty(Q)}]$; 
moreover,
\begin{align*}
| \{ x : T_{\psi^\epsilon}(\psi(x)) \neq 1 \}| 
\leq
| \{ x : \psi(x) < \epsilon \} |
\leq
C \epsilon^{\frac23};
\end{align*}
\item[($iv$)]
$\psi^\epsilon \in W^{2,\infty}(Q)$, with
\begin{align*}
\| \psi^\epsilon \|_{W^{1,\infty}(Q)} 
\leq
C,
\quad
\| \psi^\epsilon \|_{W^{2,\infty}(Q)} 
\leq
C \epsilon^{-1}.
\end{align*}
\end{itemize}
\end{lem}

\begin{proof}
($i$) We have
$\psi^\epsilon>0$ since $\rho^\epsilon \geq \epsilon/2 > 0$ and $T_{\psi}(r)$ is strictly positive for $r \in (0,1)$.
Moreover, by definition of $\rho^\epsilon$ the quantity $\rho^\epsilon(\psi(x))$ is constant on $\{ x: \psi(x) < \epsilon/2\}$, hence so is $\psi^\epsilon$.

($ii$)
We use that $\rho^\epsilon(r)=r$ for $r \geq \epsilon$, and $T_{\psi}(r)$ is strictly positive and uniformly bounded for $r \in (0,1)$ (cf. \cite[Lemma 3.1]{ElZl19}).

($iii$)
Let $r \geq r_\epsilon$.
By point ($ii$), there exists $r' \geq \epsilon$ such that $\{\psi^\epsilon = r\} = \{\psi = r'\}$. 
Hence, by definition of $\psi^\epsilon$ and $\rho^\epsilon$, we have on $\{ \psi^\epsilon = r\}$
\begin{align*}
\nabla \psi^\epsilon
=
T_{\psi}(\rho^\epsilon(\psi)) \nabla (\rho^\epsilon(\psi))
=
T_{\psi}(\psi) \nabla \psi.
\end{align*}
Therefore, the travelling time $T_{\psi^\epsilon}$ satisfies
\begin{align*}
T_{\psi^\epsilon}(r)
&:=
\int_{\{\psi^\epsilon=r\}} \frac{1}{|\nabla \psi^\epsilon|} d\gamma
=
\int_{\{\psi=r'\}} \frac{1}{T_{\psi}(\psi)|\nabla \psi|} d\gamma
=
1.
\end{align*}
As a consequence, we also have $\{T_{\psi^\epsilon}(\psi(x)) \neq 1\} \subset \{ \psi(x) < \epsilon\}$, and thus we only have to bound the Lebesgue measure of the latter set.
Since $\{\psi \geq \epsilon\}$ is convex (see \cite[Lemma 3.1]{ElZl19}), the square $Q_\epsilon$ with vertices
\begin{align*}
q \in \{ x=(x^1,x^2) \in Q : \sin(\pi x^1) = \sin(\pi x^2) = \epsilon^{\frac23} \} \subset \{\psi \geq \epsilon\}
\end{align*}
is entirely contained in the super-level set $\{\psi \geq \epsilon\}$; therefore, $|\{ \psi(x) < \epsilon\}| \leq |Q \setminus Q^\epsilon|$.
We control the latter quantity with four times the area of the strip $\{ \sin(\pi x^1) \in (0,\epsilon^{\frac23}),\, x^1 \leq 1/2\}$, which is easily computed noticing 
\begin{align*}
\pi x^1 \leq C \sin(\pi x^1) \leq C \epsilon^{\frac23}.
\end{align*}

($iv$)
Let $\psi^\star$ be defined as in \cite[Equation (3.2)]{ElZl19} (for the particular value $\alpha=1/2$), namely $\psi^\star(x) := \int_0^{\psi(x)} T_\psi(r) dr$.
Recall from the proof of \cite[Lemma 3.3]{ElZl19} that $D^2 \psi^\star$ is bounded away from $\partial Q$.
Since $\psi^\epsilon$ coincides with $\psi^\star$ on the set $\{ \psi(x) \geq \epsilon\}$, we only need to check the estimates in a neighbourhood of $\partial Q$.
Moreover, having $\psi^\epsilon(x) = \int_0^{\epsilon/2} T_{\psi}(r)dr$ identically on $\{\psi(x) \leq \epsilon/2\}$, we can restrict ourselves to controlling the Sobolev norms of $\psi^\epsilon$ on the domain
\begin{align*}
D_\epsilon := \{ x \in Q : \psi(x) \in (\epsilon/3,1/2) \}.
\end{align*}

We have
\begin{align*}
\| \psi^\epsilon \|_{W^{1,\infty}(D_\epsilon)} 
&\leq
\| T_{\psi} \|_{L^{\infty}(\epsilon/3,1/2)} 
\| \rho^\epsilon(\psi) \|_{W^{1,\infty}(D_\epsilon)}, 
\\
\| \psi^\epsilon \|_{W^{2,\infty}(D_\epsilon)} 
&\leq
\|T'_{\psi}\|_{L^{\infty}(\epsilon/3,1/2)} 
\| \rho^\epsilon(\psi) \|_{W^{1,\infty}(D_\epsilon)}^2
\\
&\quad+
\|T_{\psi}\|_{L^{\infty}(\epsilon/3,1/2)} 
\| \rho^\epsilon(\psi)\|_{W^{2,\infty}(D_\epsilon)}.
\end{align*}

By construction it holds $\| \rho^\epsilon \|_{W^{1,\infty}(\epsilon/3,1/2)} \leq C$ and $\| \rho^\epsilon \|_{W^{2,\infty}(\epsilon/3,1/2)} \leq C \epsilon^{-1}$; in addition, by \cite[Lemma 3.1]{ElZl19} we have
\begin{align*}
\|T'_{\psi}\|_{L^{\infty}(\epsilon/3,1/2)}
\leq
C \epsilon^{-\frac23},
\quad
\| \psi \|_{W^{1,\infty}(D_\epsilon)}
&\leq
C,
\quad
\| \psi \|_{W^{2,\infty}(D_\epsilon)}
\leq
C \epsilon^{-\frac13},
\end{align*}
where the last inequality is justified by the fact that $\psi \geq \epsilon/2$ implies $\sin(\pi x)+\sin(\pi y) \geq 2(\epsilon/2)^{\frac23}$. 
\end{proof}

We are ready to construct the velocity field of \autoref{thm:main}.
More precisely, we construct $v = v(t,x) := -u(T-t,x)$ as in the following:
\begin{prop} \label{prop:v_delta}
For every $k \in \N$, $\alpha \in (0,1)$ and $\delta>0$ there exists an incompressible velocity field $v \in C^k([0,T], C^\alpha_x) \cap C_{loc}([0,T), W^{1,\infty}_x)$ and a measurable $A$ with Lebesgue measure $|A| \geq 1-\delta$ such that the unique solution of $Y^\delta(t,x) = x + \int_0^t v(s,Y^\delta(s,x)) ds$ satisfies
\begin{align*}
Y^\delta(T,x) = Y_T(x),
\quad
\forall x \in A,
\end{align*}
where $Y_T$ is the map constructed in \autoref{sec:bounded}.
\end{prop}

\begin{proof}
Let us define $v^\epsilon := \nabla^\perp \psi^\epsilon$.
By construction, the particles moving with the velocity $v^\epsilon$ run across the contour lines $\{ \psi = r \}$ in time $1$ for every $r \geq \epsilon$. 
In particular, by fourfold symmetry $v^\epsilon$ rotates the super-level set $\{\psi \geq \epsilon\}$ of $\pi$ radiants in time $t=1/2$.

Let $\chi$ be a smooth function with support compactly contained in $(0,1/2)$ and satisfying $\int_0^{1/2} \chi(t)dt=1$.
In order to obtain a velocity which is smooth with respect to time, 
we define $\hat{v}^\epsilon:(0,1/2) \times Q \to \R^2$ as
\begin{align*}
\hat{v}^\epsilon (t,x) := \chi(t) v^\epsilon(x).
\end{align*} 
Then, by construction also $\hat{v}^\epsilon$ rotates the super-level set $\{\psi \geq \epsilon\}$ of $\pi$ radiants in time $t=1/2$.

Therefore, by point ($iii$) of \autoref{lem:psieps}, the rescaled velocity
\begin{align*}
v^\epsilon_n (t,x)
:= 
\frac{1}{(2\tau)^n}
\hat{v}^\epsilon ({t}/{\tau^n}, 2^n x),
\quad
n \in \N,
n \geq 1,
\end{align*}
rotates at least a fraction $1-C \epsilon^{\frac23}$ of the rescaled square $2^{-n}Q$ of an angle of $\pi$ radiants in time $t=\tau^n/2$.

Since $\psi^\epsilon$ is constant in a neighbourhood of $\partial Q$, the velocity $v^\epsilon$ vanishes at $\partial Q$. Therefore, we can glue together many (properly rescaled and reflected) copies of $v^\epsilon_n$ as in \autoref{fig:Depauw} to obtain a H\"older continuous velocity $\hat{v}^\epsilon_n : (\tau_{n-1},\tau_n) \times Q \to \R^2$ satisfying $\hat{v}^\epsilon_n(\tau_{n-1},\cdot) = \hat{v}^\epsilon_n(\tau_n,\cdot) = 0$, $\hat{v}^\epsilon_n(\cdot,x) = 0$ for $x$ in a neighbourhood of $\partial Q$, and the bounds
\begin{align*}
\| \hat{v}^\epsilon_n \|_{C_t W^{1,\infty}(Q)}
&\leq
\frac{1}{\tau^n} 
\| \chi \|_{C_t}
\| v^\epsilon \|_{W^{1,\infty}(Q)}
\leq
C'
\frac{1}{\tau^n} \epsilon^{-1},
\\
\| \partial_t^h \hat{v}^\epsilon_n \|_{C_t C^\alpha(Q)}
&\leq
\left( \frac{2^{\alpha-1}}{\tau^{1+h}}\right)^n
\| \partial_t^h \chi \|_{C_t}
\| v^\epsilon \|_{C^\alpha(Q)}
\leq
C' 
\left( \frac{2^{\alpha-1}}{\tau^{1+h}}\right)^n \epsilon^{-\alpha},
\end{align*}
for every $h \in \N$, $h \leq k$ and some finite constant $C'$ depending only on $\chi$, $k$, and the constant $C$ of \autoref{lem:psieps}.
The last inequality comes from interpolation between $W^{1,\infty}(Q)$ and $W^{2,\infty}(Q)$ Sobolev norms of $\psi^\epsilon$ given by point ($iv$) of \autoref{lem:psieps}.

Choosing $\epsilon = \epsilon_n := \epsilon_\star \kappa^n$ with $\kappa \in (0,1)$ and $\epsilon_\star \ll 1$, the previous inequalities become 
\begin{align*}
\| \hat{v}^{\epsilon_n}_n \|_{C_t W^{1,\infty}(Q)}
&
\leq
C' \epsilon_\star^{-1} \left( \frac{1}{\tau \kappa} \right)^n ,
\\
\| \partial_t^h \hat{v}^{\epsilon_n}_n \|_{C_t C^\alpha(Q)}
&\leq
C' \epsilon_\star^{-\alpha}
\left( \frac{2^{\alpha-1}}{\tau^{1+h} \kappa^\alpha}\right)^n 
.
\end{align*}
Therefore, we take
\begin{align*}
\tau^{1+k} = \kappa^\alpha = 2^{\frac{\alpha-1}{3}} \in (0,1),
\end{align*}
so that $\hat{v}^{\epsilon_n}_n$ tends to $0$ in $C^k_t C^\alpha(Q)$ as $n \to \infty$.
In this way, the divergence-free velocity field $v := \sum_{n \geq 1} \hat{v}^{\epsilon_n}_n \mathbf{1}_{\{ t \in (\tau_{n-1},\tau_n)\}}$, extended periodically on the full torus $\T^2$, is in $C^k([0,T],C^\alpha(\T^2)) \cap C_{loc}([0,T),W^{1,\infty}(\T^2))$ for $T := \sum_q \tau^q < \infty$.

Finally, $v \in C_{loc}([0,T),W^{1,\infty}(\T^2))$ with null divergence implies that there exists a unique flow $Y^\delta$ associated to $v$, which preserves the Lebesgue measure for every $t<T$.
By ($iii$) of \autoref{lem:psieps}, on every time interval of the form $[\tau_{n-1},\tau_{n-1} + \tau^n/2)$ (resp. $[\tau_{n-1} + \tau^n/2,\tau_n)$) the flow $Y^\delta$ coincides with $Y$ on a measurable set $A^1_n$ (resp. $A^2_n$) of Lebesgue measure at least
\begin{align*}
|A^1_n| = |A^2_n| \geq 1- \frac{C}{2} \epsilon_\star^{\frac23} \kappa^{\frac{2n}{3}}.
\end{align*}
Defining
\begin{align*}
A := \bigcap_{n \geq 1} 
(Y^\delta_{\tau_{n-1}})^{-1} (A^1_n)
\cap
(Y^\delta_{\tau_{n-1} + \tau^n/2})^{-1} (A^2_n),
\end{align*}
we have that $Y^\delta_T$ coincides with $Y_T$ on $A$ and 
\begin{align*}
|A| 
\geq 
1- C \epsilon_\star^{\frac23} \sum_{n} \kappa^{\frac{2n}{3}}
\geq  
1-\delta,
\end{align*}
up to choosing $\epsilon_\star = \epsilon_\star(\alpha,\delta)$ sufficiently small.
\end{proof}

\begin{rmk}
Notice that $v^\epsilon \in C^1_x$ except at most at the center of the square $Q$, where it vanishes.  
With only minor modifications to the definition of the streamfunction $\psi^\epsilon$, we can make $v^\epsilon \equiv 0$ in an arbitrary small neighbourhood of the center of $Q$, improving the $W^{1,\infty}_x$ bound on $v$ to an actual $C^1_x$ bound (up to replacing the bound on $|A|$ with $|A| \geq 1-2\delta$). 
In particular, the incompressibility condition in \autoref{prop:v_delta} (and in \autoref{thm:main}) can be understood in the strong analytic sense.
\end{rmk}

\subsection{Proof of \autoref{thm:main}}
We finally give the proof of our main result.
\begin{proof}[Proof of \autoref{thm:main}]
Let $v$ and $A$ as in \autoref{prop:v_delta}, and define $u(t,x) := -v(T-t,x)$.
Let us suppose per absurdum there exists a measure-preserving selection of characteristic $X$ for the velocity field $u$.

Let $x(U):=U_x \subset \T^2$ be given by the image of $U$ under the map $x:\Omega \to \T^2$ defined in \autoref{ssec:dyadic}.
$U_x$ is a Borel set with full Lebesgue measure.
Since $X_T$ is Lebesgue measurable we can write 
\begin{align*}
X_T^{-1}(\{ \omega^1_1 = 0 \} \cap A) = B_0 \cup N_0,
\\
X_T^{-1}(\{ \omega^1_1 = 1 \} \cap A) = B_1 \cup N_1,
\end{align*}
where $B_0,B_1$ are Borel subsets of $U_x \subset \T^2$ and $N_0,N_1$ are Lebesgue measurable with $|N_0|=|N_1|=0$.
Since we are assuming $X_T$ measure-preserving, it holds $|B_0 \cup B_1| = |A| \geq 1-\delta$.
Moreover, without loss of generality we may also assume that $X(\cdot,x)$ is a solution of the ODE \eqref{eq:ODE} for every $x \in B_0 \cup B_1$.
By uniqueness of solutions to \eqref{eq:back}, on the time interval $t\in[0,T]$ we have for every $x \in B_0 \cup B_1$
\begin{align*}
X(T-t,x) = Y^\delta(t,X_T(x)).
\end{align*}

Since $Y^\delta_T|_{A \cap U_x} = Y_T|_{A \cap U_x}:A \cap U_x \to U_x$ is Borel measurable (as pointwise limit as $n\to \infty$ of the Borel measurable maps $Y_{\tau_n}$) we have that for $i \in \{0,1\}$ the set
\begin{align*}
C_i := X_T(B_i) \cap A \cap U_x = (Y_T|_{A \cap U_x})^{-1} (B_i) \cap \{ \omega^1_1 = i\} 
\end{align*}
is Borel measurable.
Then we have
\begin{align*}
X_T(B_0 \sqcup B_1) \cap A \cap U_x
=
C_0 \sqcup C_1
\end{align*} 
where $\sqcup$ denotes the disjoint union since $X_T(B_i) \subset \{\omega^1_1 = i\}$.
Therefore, denoting $\sigma_x:U_x \to U_x$ the map conjugate to the digits change map $\sigma : U \to U$ defined by \eqref{eq:sigma}, we have   
\begin{align*}
|X_T(B_0 \sqcup B_1) \cap A \cap U_x|
&=
|C_0|
+
|C_1|
\\
&=
|\sigma_x (C_0)|
+
|\sigma_x (C_1)|,
\end{align*}
since the map $\sigma$ is measure-preserving by \autoref{lem:measure}.
Moreover, by \eqref{eq:Psi_sigma} it holds $Y_T \sigma_x (x) = Y_T (x)$ for every $x \in U_x$; therefore, the sets
\begin{align*}
C_0,\quad
C_1,\quad
\sigma_x (C_0),\quad
\sigma_x (C_1),
\end{align*}
are pairwise disjoint. Indeed, $\sigma_x C_0, \sigma_x C_1$ are disjoint (because $Y_T \sigma_x = Y_T$ on $U_x$, and $Y_T C_0 \cap Y_T C_1 = \emptyset$) and are disjoint from $X_T(B_0 \sqcup B_1) \cap A \cap U_x$, otherwise
\begin{align*}
y = X_T(x) = \sigma_x (X_T(x'))
\Rightarrow
Y_T(y) = Y_T(X_T(x)) = Y_T (\sigma_x (X_T(x')))
\Rightarrow
x = x'
\end{align*}
but $X_T(x) \neq \sigma_x (X_T(x))$ since $\sigma_x$ has no fixed point in $U_x$.
Therefore,
\begin{align*}
\sigma_x (C_0) \sqcup \sigma_x (C_1)
&=
\sigma_x (C_0 \sqcup C_1)
=
\sigma_x(X_T(B_0 \sqcup B_1) \cap A \cap U_x);
\end{align*}
in particular, $X_T(B_0 \sqcup B_1) \cap A \cap U_x$ and $\sigma_x (X_T(B_0 \sqcup B_1) \cap A \cap U_x)$ are disjoint sets with the same Lebesgue measure, implying (recall $U_x$ has full measure)
\begin{align*}
|X_T(B_0 \sqcup B_1) \cap A| = |X_T(B_0 \sqcup B_1) \cap A \cap U_x| \leq 1/2.
\end{align*} 
But this gives a contradiction for $\delta <1/4$,  since
\begin{align*}
|X_T(B_0 \sqcup B_1)|
=
|X_T(B_0 \sqcup B_1) \cap A|
+
|X_T(B_0 \sqcup B_1) \cap A^c|
\leq
1/2+\delta,
\end{align*}
but
\begin{align*}
1-\delta
\leq
|B_0 \sqcup B_1| 
\leq 
|X_T^{-1} (X_T(B_0 \sqcup B_1))|
=
|X_T(B_0 \sqcup B_1)|
\leq 
1/2+\delta.
\end{align*}
\end{proof}

\bibliographystyle{alpha}

\begin{thebibliography}{BCDL21}

\bibitem[Amb04]{Am04}
Luigi Ambrosio.
\newblock Transport equation and {C}auchy problem for {BV} vector fields.
\newblock {\em Inventiones mathematicae}, 158(2):227--260, 2004.

\bibitem[BB20]{BiBo20}
Stefano Bianchini and Paolo Bonicatto.
\newblock A uniqueness result for the decomposition of vector fields in
  $\mathbb{R}^d$.
\newblock {\em Inventiones mathematicae}, 220:255--393, 2020.

\bibitem[BCDL21]{BrCoDL20}
Elia Bruè, Maria Colombo, and Camillo De~Lellis.
\newblock Positive solutions of transport equations and classical nonuniqueness
  of characteristic curves.
\newblock {\em Arch. Rat. Mech. Anal.}, 240:1055--1090, 2021.

\bibitem[CC21]{CaCr21}
Laura Caravenna and Gianluca Crippa.
\newblock A directional {L}ipschitz extension lemma, with applications to
  uniqueness and {L}agrangianity for the continuity equation.
\newblock {\em Communications in Partial Differential Equations},
  46(8):1488--1520, 2021.

\bibitem[Dep03]{De03}
Nicolas Depauw.
\newblock Non unicit\'e des solutions born\'ees pour un champ de vecteurs {BV}
  en dehors d'un hyperplan.
\newblock {\em Comptes Rendus. Math\'ematique}, 337(4):249--252, 2003.

\bibitem[DL89]{DiLi89}
Ronald~J. Diperna and Pierre-Louis Lions.
\newblock Ordinary differential equations, transport theory and {S}obolev
  spaces.
\newblock {\em Inventiones mathematicae}, 98:511--547, 1989.

\bibitem[EZ19]{ElZl19}
Tarek~M. Elgindi and Andrej Zlatoš.
\newblock Universal mixers in all dimensions.
\newblock {\em Advances in Mathematics}, 356:106807, 2019.

\bibitem[Hal50]{Ha50}
Paul~R. Halmos.
\newblock {\em Measure Theory}.
\newblock Graduate Texts in Mathematics 18. Springer-Verlag New York, 1st edition, 1950.

\bibitem[Kum23]{Ku23+}
Anuj Kumar.
\newblock Nonuniqueness of trajectories on a set of full measure for {S}obolev
  vector fields.
\newblock arXiv:2301.05185, 2023.

\bibitem[PS23]{PiSo23}
Jules Pitcho and Massimo Sorella.
\newblock Almost everywhere nonuniqueness of integral curves for
  divergence-free {S}obolev vector fields.
\newblock {\em SIAM Journal on Mathematical Analysis}, 55(5):4640--4663, 2023.

\bibitem[Ziz22]{Zi22}
Martina Zizza.
\newblock An example of a weakly mixing {BV} vector field which is not strongly mixing. 
\newblock arXiv:2208.12174, 2022.

\end{thebibliography}

\end{document}